\documentclass [12pt,a4paper]{amsart}

\usepackage{amssymb,amsmath,amsthm}

\usepackage{url}

\begin{document}

\newtheorem{theorem}{Theorem}[section]
\newtheorem{lemma}[theorem]{Lemma}
\newtheorem{proposition}[theorem]{Proposition}

\newtheorem{statement}[theorem]{Statement}
\newtheorem{conjecture}[theorem]{Conjecture}
\newtheorem{corollary}[theorem]{Corollary}

\newtheorem*{maintheorem}{Theorem A}

\newtheoremstyle{neosn}{0.5\topsep}{0.5\topsep}{\rm}{}{\bf}{.}{ }{\thmname{#1}\thmnumber{ #2}\thmnote{ {\mdseries#3}}}
\theoremstyle{neosn}
\newtheorem{remark}[theorem]{Remark}
\newtheorem{definition}[theorem]{Definition}
\newtheorem{example}[theorem]{Example}

\numberwithin{equation}{section}

\newcommand{\Aut}{\,\mathrm{Aut}\,}
\newcommand{\Inn}{\,\mathrm{Inn}\,}
\newcommand{\End}{\,\mathrm{End}\,}
\newcommand{\Out}{\,\mathrm{Out}\,}
\newcommand{\Hom}{\,\mathrm{Hom}\,}
\newcommand{\ad}{\,\mathrm{ad}\,}
\newcommand{\SLn}{$\rm sl(n)$}
\newcommand{\calL}{{\mathcal L}}
\renewenvironment{proof}{\noindent \textbf{Proof.}}{$\blacksquare$}
\newcommand{\GL}{\,\mathrm{GL}\,}
\newcommand{\SL}{\,\mathrm{SL}\,}
\newcommand{\PGL}{\,\mathrm{PGL}\,}
\newcommand{\PSL}{\,\mathrm{PSL}\,}
\newcommand{\PSO}{\,\mathrm{PSO}\,}
\newcommand{\SO}{\,\mathrm{SO}\,}
\newcommand{\Sp}{\,\mathrm{Sp}\,}
\newcommand{\PSp}{\,\mathrm{PSp}\,}
\newcommand{\UT}{\,\mathrm{UT}\,}
\newcommand{\Exp}{\,\mathrm{Exp}\,}
\newcommand{\Rad}{\,\mathrm{Rad}\,}
\newcommand{\charr}{\mathrm{char}\,}
\newcommand{\tr}{\mathrm{tr}\,}
\newcommand{\diag}{{\rm diag}}

\font\cyr=wncyr10 scaled \magstep1%

\def\Sha{\text{\cyr Sh}}

\title{\Sha-rigidity of Chevalley groups over local rings}


\author{Elena Bunina, Boris Kunyavski\u\i }
\date{}
\address{Department of Mathematics, Bar--Ilan University, 5290002 Ramat Gan, ISRAEL}
\email{helenbunina@gmail.com}
\email{kunyav@gmail.com}

\thanks{Research of Boris Kunyavski\u\i{} was supported by the ISF grant 1994/20.
A part of this research was accomplished when he was visiting the IHES (Bures-sur-Yvette). Support of these institutions is gratefully acknowledged}

\begin{abstract}
We prove that every locally inner endomorphism of a Chevalley group (or its elementary subgroup) over a local ring with an
irreducible root system of rank $>1$ (with $1/2$ for the systems $\mathbf A_2$, $\mathbf F_4$, $\mathbf B_\ell$, $\mathbf C_\ell$
and with $1/3$ for the system~$\mathbf G_2$) is inner, so that all these groups are \Sha-rigid.
\end{abstract}

\maketitle

\section{Introduction}\leavevmode
Our main focus in the present paper is on a local--global invariant $\Sha(G)$ of a group $G$.
Here the Cyrillic letter \Sha (``Sha'') refers to one of the names of this invariant, the \emph{Shafarevich--Tate set}.

For a finite group~$G$, the invariant we are talking about was introduced by Burnside
(naturally, under a different name) as early as in 1911, in the second edition of his famous
book~\cite{BK21}. Soon enough, in the paper~\cite{BK22}, published in 1913, he constructed
the first example of~$G$ with nontrivial $\Sha(G)$. Since then, $\Sha (G)$ has been
rediscovered more than once, each time revealing some new features.

Our main result (Theorem A) asserts that for any Chevalley group $G$ over a local ring $R$  (under mild restrictions
on invertible elements of $R$) the set $\Sha (G)$ consists of one element. In~\cite{BK}, such groups were called
\Sha-\emph{rigid}.

We start with a cohomological definition of $\Sha (G)$ introduced by Takashi Ono~\cite{BK112}.

\begin{definition}[\cite{BK112}] \label{def:ono}
Let $G$ be a group acting on itself by conjugation, $(g, x)\mapsto  gxg^{-1}$,
and let $H^1(G, G)$ denote the first cohomology pointed set. The set of cohomology classes
becoming trivial after restriction to every cyclic subgroup of~$G$ is denoted by $\Sha (G)$ and
called the \emph{Shafarevich--Tate set} of~$G$.
\end{definition}

\begin{definition}
For the lack of a better term, we say that a group $G$ with one-element
Shafarevich--Tate set is \emph{\Sha-rigid}.
\end{definition}

This term will be explained later, after clarifying relationships with some rigidity
phenomena.
Note that \Sha-rigidity is often a crucial step to establishing important properties of~$G$,
or of a whole class of groups. On the other hand, groups with nontrivial $\Sha (G)$ often
provide interesting examples (or even allow one to refute long-standing conjectures). Some
instances will be given below.

\begin{remark}
The terminology of Definition~\ref{def:ono} originates in the prototype of $\Sha (G)$,
dating back to the 1950’s when it appeared in the context of a high-brow approach to
diophantine equations and has been remaining since then one of the favourite objects of
arithmetic geometers: given an algebraic group $A$ defined over a number field~$k$, $\Sha (A)$ is
defined as the set of cohomology classes $H^1(\Gamma, A(\overline k))$ (where the absolute Galois group
$\Gamma=Gal(\overline k/k)$ acts naturally on $\overline k$-points of~$A$) that become trivial after restriction to every
$\Gamma_v = Gal(\overline k_v/k_v)$, where $v$ runs over all places of~$k$. In the purely group-theoretic setting
as above, a much more down-to-earth description is available.
\end{remark}

The following important remark is due to Marcin Mazur, see~\cite{BK115}
(note that it appeared implicitly in an earlier paper by Chih-Han Sah \cite{Sah}).

\begin{remark}
A map $f \colon G\to G$ is a cocycle if and only if the map $g\colon G\to G$ defined
by $g(x) = f (x)x$ is an endomorphism. Furthermore, $f$ is a coboundary if and only if $g$ is
an inner automorphism, and the restriction of $f$ to the cyclic subgroup generated by $x\in G$
is a coboundary if and only if $g(x)$ is conjugate to~$x$. Denote by $\End_c(G)$ (resp. $\Aut_c(G)$)
the set of endomorphisms (resp. automorphisms) $g$ of~$G$ such that $g(x)$ is conjugate to~$x$
for all $x\in G$.
\end{remark}

We see that $G$ is \Sha-rigid if and only if it satisfies the following condition:
$$
\End_c(G) = \Inn(G).\eqno (*)
$$

\begin{remark}
Endomorphisms (or automorphisms) $g$ with the property that $g(x)$ is
conjugate to~$x$ for all $x\in G$ appear in the literature under different names: \emph{pointwise
inner}, \emph{conjugating}, \emph{class-preserving}, etc.

In this text they will be called \emph{locally inner}. Note
that any locally inner endomorphism is injective.
\end{remark}
Condition  $(*)$ is sometimes called \emph{Property}~E (see, e.g.,~\cite{BK5}).
Since $\Inn(G)\subseteq \Aut_c(G)\subseteq \End_c(G)$, it is convenient to subdivide the property of being
\Sha-rigid into two weaker ones:

(A) $\Inn(G)=\Aut_c(G)$;

(B) $\Aut_c(G)=\End_c(G)$.

The first property is sometimes referred to as \emph{Property}~A (see, e.g.,~\cite{BK57}).
It can be written down as $\Out_c(G)=1$ meaning that $G$ has no locally inner outer automorphisms.

In this text, a group $G$ satisfying~(A) will be called A-\emph{rigid}, and a group satisfying~(B) will be called
B-\emph{rigid}. In these terms, any \Sha-rigid group is both A-rigid and B-rigid, and vice versa.

There are numerous examples of A-, B-, and \Sha-rigid groups.
To convince the reader in the diversity of thereof, we reproduce
a list from \cite{BK} (see a survey \cite{Y2} and more recent papers \cite{KK}, \cite{Ra}
for more examples; Lie-algebraic analogues are discussed in \cite{Ku2}, \cite{KO}).

\begin{example} \label{ex1}
The following groups are B-rigid:
(1) finite;
(2) profinite;
(3) solvable;
(4) co-Hopfian.
\end{example}

\begin{example} \label{ex2}
The following finite groups are A-rigid:
(1) symmetric groups;
(2) simple groups;
(3) $p$-groups of order at most $p^4$;
(4) $p$-groups having a maximal cyclic subgroup;
(5) extraspecial and almost extraspecial $p$-groups;
(6) $p$-groups having a cyclic subgroup of index $p^2$;
(7) groups such that the Sylow $p$-subgroups are cyclic for odd~$p$, and either cyclic, or
dihedral, or generalized quaternion for $p=2$;
(8) Blackburn groups;
(9) abelian-by-cyclic groups;
(10) primitive supersolvable groups;
(11) unitriangular matrix groups over $\mathbb F_p$ and the quotients of their lower central
series;
(12) central products of A-rigid groups.
\end{example}

\begin{example} \label{ex3}
The following infinite groups are A-rigid:
(1) the absolute Galois group of~$\mathbb Q$;
(2) the absolute Galois group of $\mathbb Q_p$;
(3) non-abelian free groups;
(4) non-abelian free profinite groups;
(5) so-called pseudo-p-free profinite groups;
(6) free nilpotent groups;
(7) non-abelian free solvable groups;
(8) nontrivial free products;
(9) one-relator groups of the form $\langle a, b\mid [a^m, b^n] = 1\rangle$ , $m, n > 1$;
(10) non-abelian free Burnside groups of large odd exponent;
(11) fundamental groups of compact orientable surfaces;
(12) Artin braid groups $B_n$ and pure braid groups $P_n$;
(13) connected compact topological groups;
(14) fundamental groups of closed surfaces with negative Euler characteristic;
(15) non-elementary subgroups $H$ of hyperbolic groups~$G$ such that $H$ does not
normalize any nontrivial finite subgroup of~$G$;
(16) some groups of automorphisms and birational automorphisms of the plane and
space;
(17) unitriangular matrix groups over~$\mathbb Q$ and the quotients of their lower central series,
as well unitriangular matrix groups over~$\mathbb Z$;
(18) all finitely generated Coxeter groups.
\end{example}

\begin{example} \label{ex4}
The following  groups are \Sha-rigid:
(1) groups appearing in Example~\ref{ex1} simultaneously  with one of Examples~\ref{ex2} or \ref{ex3}.
(2) free groups;
(3) groups $\SL_n(R)$, $\PSL_n(R)$ and $\GL_n(R)$ where $R$ is a euclidean domain;
(4) free products of at least two nontrivial groups;
(5) amalgamated products $A*_H B$ where $H$ is a maximal cyclic subgroup of~$A$ and~$B$;
(6) all Fuchsian groups $G(n, r, s)$ except, possibly, triangle groups $G(0, 0, 3)$;
(7) almost all orientable Seifert groups, except possibly $G_1(0, 3)$ and $G_1(1, 1)$;
(8) some ``polygonal''  and ``tree'' products;
(9) the Cremona group of birational automorphisms of the complex projective
plane;
(10) all torsion-free hyperbolic groups;
(11) right angled Artin groups.
\end{example}

As to groups that are not \Sha-rigid, one can find many such examples among finite $p$-groups. 
As mentioned above, the first such example was discovered by Burnside. The smallest non-\Sha-rigid 
group is of order 32, see \cite{WalG}. To give the reader some feeling of the situation for 
infinite groups, we present below a ``generic'' construction of non-\Sha-rigid groups, which provides 
an additional argument in favour of our choice of the name for locally inner endomorphisms. The idea 
is attributed to Passman, see \cite[Introduction]{Sah}. 

\begin{example}
Let $G={\mathrm{FSym}}(\Omega )$ be a finitary symmetric group (the group of all permutations of an infinite set $\Omega$
fixing all but finitely many elements of $\Omega$). Viewing $G$ as a subgroup of the symmetric group ${\mathrm{Sym}}(\Omega)$, 
consider an automorphism $\varphi\colon G\to G$ induced by conjugation by some $a\in {\mathrm{Sym}}(\Omega)\setminus {\mathrm{FSym}}(\Omega )$. 
Clearly, $\varphi$ is almost inner but not inner, hence $G$ is not A-rigid, thus not \Sha-rigid. 
\end{example}

With an eye to extending the list of Example \ref{ex4}, the following conjecture
was proposed in~\cite{BK}.

\begin{conjecture} \label{conj:main}
 The following groups $G$ are \Sha-rigid:

\emph{(1)} $G=G(k)$, the group of $k$-points of a split simple Chevalley group $G$ defined over a
sufficiently large field~$k$;
\emph{(2)} the same as in~\emph{(1)}, with $k$ replaced with some ``good'' ring;
\emph{(3)} the same as in~\emph{(1)}, with any isotropic $k$-group~$G$;
\emph{(4)} the same as in~\emph{(1)}, with $G$ an anisotropic group splitting over a quadratic extension of~$k$;
\emph{(5)} $G\subset G(k)$ is a ``big'' subgroup possessing some rigidity properties in the sense of
Mostow, Margulis, and others;
\emph{(6)} $G$ is a split Kac--Moody group over a sufficiently large field~$k$.
\end{conjecture}

In this paper we prove assertions (1), (2) and partially (5), taking local rings as ``good'' rings.
Here is our main theorem.

\begin{maintheorem}\label{main_theorem}
Let $\Phi$ be a reduced irreducible root system of rank $>1$, $R$ be a local ring
with $1/2$ for $\Phi=\mathbf A_2, \mathbf B_\ell, \mathbf C_\ell, \mathbf F_4$
and with $1/3$ for $\Phi=\mathbf G_2$. Then any Chevalley group $G(\Phi,R)$
is \Sha-rigid, and so is its elementary subgroup $E(\Phi,R)$.
\end{maintheorem}

The proof is contained in Section \ref{sec:proof}. It heavily relies on the classification
of automorphisms obtained by the first author in a series of papers, see~\cite{Bunina_main}
and the references therein.

\begin{remark}
The restriction on the rank can be lifted in the case where $R$ is a local Euclidean domain,
thanks to the results of Ono reproduced (in greater generality) in \cite{Wa1}, \cite{Wa2}.
\end{remark}

\begin{remark}
In the case where $R$ is a {\it finite field}, the assertion of Theorem A,
without any restrictive assumptions, follows from
the Feit--Seitz theorem on the A-rigidity of finite simple groups \cite{FS}, proved
using the classification of thereof. If $R$ is an {\it arbitrary field}, Theorem A is
new, to the best of our knowledge (though it can perhaps be deduced from Steinberg's
classical results).
\end{remark}

In the next section we collect the needed preliminaries on Chevalley groups, particularly over local rings.
Section \ref{sec:proof} contains the proof of Theorem A. In Section \ref{sec:final} we discuss some open problems
and eventual generalizations.

\section{Chevalley groups over local rings}\leavevmode \label{sec:Chev}
First fix some notation. It is standard and mainly follows
\cite{v43} and \cite{VavPlotk1} where the reader can find additional details and further references
on Chevalley groups over rings.

Given a reduced irreducible root system $\Phi$
and a commutative local ring~$R$ with~$1$,
we assume throughout that for the root systems ${\mathbf A}_2$, ${\mathbf B}_\ell$, ${\mathbf C}_\ell$ and ${\mathbf F}_4$
the ring $R$ contains $1/2$ and for ${\mathbf G}_2$ it contains $1/3$.

We consider an arbitrary Chevalley group $G_{\mathcal P}(\Phi,R)$
constructed by $\Phi$, $R$ and a lattice $\mathcal P(\Phi)$ lying between
the root lattice $Q(\Phi)$ and weight lattice $P(\Phi)$. In the extreme cases
$\mathcal P=P$ or $\mathcal P=Q$ we say that $G_{\mathcal P}(\Phi,R)$ is
simply connected or adjoint and denote it by $G_{\mathrm{sc}}(\Phi,R)$ or
$G_{\mathrm{ad}}(\Phi,R)$, respectively.

Whenever any of $\Phi$, $R$ and $\mathcal P$ is fixed, we will omit it in the notation;
thus we often shorten $G_{\mathcal P}(\Phi,R)$ to $G$. The same abuse of notation will be applied
to subgroups of $G$ if this does not lead to any confusion.


We fix a split maximal torus $T=T(\Phi,R)$ in $G=G(\Phi,R)$
and identify $\Phi$ with $\Phi(G,T)$. This choice
uniquely determines the unipotent root subgroups, $X_{\alpha}$,
$\alpha\in\Phi$, in $G$, elementary with respect to $T$. As usual,
we fix maps $x_{\alpha}\colon R\mapsto X_{\alpha}$, so that
$X_{\alpha}=\{x_{\alpha}(t)\mid t\in R\}$, and require that
these parametrizations satisfy the Chevalley commutator formula
with integer coefficients, see \cite{Carter}, \cite{Steinberg}.
The above unipotent elements $x_{\alpha}(t)$, where $\alpha\in\Phi$, $t\in R$,
elementary with respect to $T(\Phi,R)$, are called
elementary root unipotents.


Further,
$$ E(\Phi,R)=\big\langle x_\alpha(t),\ \alpha\in\Phi,\ t\in R\big\rangle $$
\noindent
denotes the elementary subgroup of $G(\Phi,R)$,
spanned by all elementary root unipotents, or, what is the
same, by all root subgroups $X_{\alpha}$,
$\alpha\in\Phi$.


If $R$ is an algebraically closed field, then
$$
G_{\mathcal P} (\Phi,R)=E_{\mathcal P}(\Phi,R)
$$
for any lattice~$\mathcal P$. This equality is not always true even $R$
is a non-algebraically closed field.
However if $G$ is simply connected 
and $R$ is \emph{semilocal}
(i.e., contains only finitely many maximal ideals), then we have
$$
G_{sc}(\Phi,R)=E_{sc}(\Phi,R),
$$
see, e.g.,  \cite{M}. 
If, however, $G$ is not simply connected,
$G_{\mathcal P}(\Phi,R)=E_{\mathcal P}(\Phi,R)T_{\mathcal P}(\Phi,R)$ (see~
\cite{M}). 

If $\Phi$ is irreducible of rank $\ell\geqslant
2$, then the Suslin--Kopeiko--Taddei theorem says that $E(\Phi,R)$ is always normal in $G(\Phi,R)$ (see~\cite{v41} and the references therein).
In the case $R$ is local
with~$1/2$ for $\mathbf B_\ell, \mathbf C_\ell, \mathbf F_4$ and with $1/3$ for~$\mathbf G_2$, it is known (see, e.g., \cite{Vas})  that
$$
[G(\Phi,R),G(\Phi,R)]=[E(\Phi,R),E(\Phi,R)]=E(\Phi,R).
$$

However, for $\ell=1$ the subgroup of elementary matrices
$E_2(R)=E_{sc}(\mathbf A_1,R)$ is not necessarily normal in the special linear
group $\SL_2(R)=G_{sc}(\mathbf A_1,R)$ (the first example is due to Paul Cohn~\cite{Cn}). 

The elementary adjoint Chevalley group $E_{\mathrm{ad}} (\Phi, R)$ is always the quotient group of $E_{\mathcal P}(\Phi, R)$ by its centre.


Choosing an order on $\Phi$,
denote by $\Phi^+$ the set of positive roots. Set $m=|\Phi^+|$.
The set of
simple roots is denoted by~$\Delta$. The subgroup
$$
U=U(G)=U(\Phi,R)
$$
is generated by the elements $x_\alpha(t)$, $\alpha\in
\Phi^+$, $t\in R$. Similarly, the subgroup $V=V(G)=V(\Phi,R)$ is generated by the elements
$x_{-\alpha}(t)$, $\alpha\in \Phi^+$, $t\in R$.

For any order of the positive roots $\alpha_1,\dots, \alpha_m$, all elements $g_1\in U(G)$ and $g_2\in V(G)$ are uniquely represented in the form
$$
g_1=x_{\alpha_1}(a_1)\dots x_{\alpha_m}(a_m)\text{ and } g_2=x_{-\alpha_1}(b_1)\dots x_{-\alpha_m}(b_m),\text{ where }a_1,\dots, a_m,b_1,\dots, b_m\in R.
$$

We will usually order the positive roots by their \emph{height}: start with simple roots $\alpha_1,\dots, \alpha_{\ell}$, then put sums of two simple roots, etc.

For  any $t\in R^*$ and any $\alpha\in \Phi$  we denote
$$
w_\alpha (t):=x_\alpha(t)x_{-\alpha}(-t^{-1}) x_\alpha (t)\text{ and }
h_\alpha(t):=w_\alpha(t)w_\alpha(1)^{-1}.
$$
For the group $H=H(G)=H(\Phi,R)$ generated by the $h_\alpha(t)$, $\alpha\in
\Phi$, $t\in R^*$, we have
$$
H(\Phi,R)=T(\Phi,R)\cap E(\Phi,R).
$$
Any element $h\in H$ is a product $h_{\alpha_1}(t_1)\dots h_{\alpha_\ell}(t_\ell)$, where $\alpha_1,\dots, \alpha_\ell$ are simple roots, $t_1,\dots, t_\ell\in R^*$.

The normalizer $N$ of  $H$ in $E$ is generated by all $w_\alpha(t)$, $\alpha\in \Phi$, $t\in R^*$, and its quotient $\mathbf W:=N/H$ is isomorphic to the Weyl group of $\Phi$.
The elements $w_{\alpha}(1)$ will be denoted by~$\mathbf w_\alpha$ and sometimes identified with the reflections $w_\alpha$ from the Weyl group.

If $R$ is a field, the groups $G(\Phi,R)$ and $E(\Phi,R)$ admit the \emph{Bruhat decomposition}: any $g\in G$ is uniquely represented as a product
\begin{equation} \label{Bruhat}
g=t x_{\alpha_1}(a_1)\dots x_{\alpha_m}(a_m)\mathbf w x_{\alpha_{i_1}}(b_1)\dots x_{\alpha_{i_k}}(b_k), \text{ where }  t\in T(G), \mathbf w\in \mathbf W, a_1,\dots, b_k \in R,
\end{equation}
and for any $j=1,\dots, k$ we have $\mathbf w (\alpha_{i_j})\in \Phi^-$. If $g\in E$, we have $t\in H$.

If $R$ is a local ring which is not a field, we use the \emph{Gauss decomposition}: any $g\in G$ can be represented as
\begin{equation}
g= t u_1 v u_2,\text{ where } t\in T, u_1,u_2\in U, v\in V.
\end{equation}
For $g\in E$ we have $t\in H(G)$. This decomposition is not unique.

We will also heavily use the standard description of automorphisms and isomorphisms of Chevalley groups over local rings.

Let us describe some types of automorphisms of Chevalley groups $G(\Phi,R)$ and $E(\Phi, R)$, which are called (\emph{basic}) \emph{standard}:

\medskip

{\bf 1.} For $x\in G(\Phi,R)$ or $x\in E(\Phi,R)$, the conjugation $i_x\colon g\mapsto xgx^{-1}$ is called an \emph{inner} automorphism of the corresponding Chevalley group.

\medskip

{\bf 2.} For $d_1,\dots, d_\ell\in R^*$ there is a unique automorphism $d$ of $G(\Phi,R)$  such that $d\colon x_{\alpha_i}(t) \mapsto x_{\alpha_i}(d_i\cdot t)$ for all $i=1,\dots, \ell$ and $t\in R$. This automorphism is called \emph{diagonal}.

It corresponds to the conjugation by an element $t\in T(\Phi,S)$ where a ring $S$ is an extension of~$R$. Thus a diagonal automorphism may be inner.

\medskip

{\bf 3.} If $\delta$ is an automorphism of the Dynkin diagram of~$\Phi$ (a substitution of simple roots preserving their lengths and angles between them),
then $\delta$ can be uniquely extended up to an automorphism $\delta$ (we will use the same letters for both of them) of $G(\Phi,R)$ or $E(\Phi,R)$
such that $\delta (x_{\alpha_i}(t))=x_{\delta(\alpha_i)}(t)$ for all $i=1,\dots, \ell$ and $t\in R$. This automorphism $\delta$ is called a \emph{graph} automorphism.

Non-trivial graph automorphisms exist only for the root systems $\mathbf A_\ell$, $\ell\geqslant 2$, $\mathbf D_\ell$, $\ell \geqslant 4$, and~$\mathbf E_6$.
For $\mathbf A_\ell$ the only non-trivial $\delta$ swaps the simple roots $\alpha_i$ and $\alpha_{\ell+1-i}$, $i=1,\dots \ell$.
For $\mathbf D_4$ a graph automorphism can realize any substitution of the roots $\alpha_1,\alpha_3,\alpha_4$.
For $\mathbf D_\ell$, $\ell\geqslant 5$, the only non-trivial $\delta$ swaps the roots $\alpha_{\ell-1}$ and $\alpha_\ell$ and does not move other simple roots.
Finally, for $\mathbf E_6$ this $\delta$ swaps $\alpha_1$ and $\alpha_6$, $\alpha_2$ and $\alpha_5$ and does not move $\alpha_3$ and $\alpha_4$.

\medskip

{\bf 4.} If $\rho$ is an automorphism of the ring~$R$, then the self-map of $G(\Phi,R)$ or $E(\Phi,R)$
taking $x_\alpha(t)$ to $x_\alpha(\rho(t))$ is a group automorphism. We call it a \emph{ring} automorphism
and denote by the same letter~$\rho$. It is clear that in any representation of $G$ it acts
by taking the matrix $(a_{ij})$ to $(\rho(a_{ij}))$.

More generally, given two Chevalley groups $G(\Phi,R)$ and $G(\Phi,S)$ (or $E(\Phi,R)$ and $E(\Phi,S)$),
an isomorphism (or even a homomorphism) of rings $\rho\colon R\to S$ induces a \emph{ring isomorphism} (or \emph{ring homomorphism})
of Chevalley groups $\rho\colon G(\Phi,R)\to G(\Phi,S)$ (or $\rho\colon E(\Phi,R)\to E(\Phi,S)$).

\medskip

{\bf 5.} Any automorphism $\Gamma$ of the Chevalley group $G$ such that  for all $g\in G$ we have
$$
\Gamma (g) =\lambda_g \cdot g,\text{ where } \lambda_g \in Z(G),
$$
is called \emph{central}. On the elementary subgroup $E$ it is identical since $E$ coincides with the commutator subgroup of $G$:
$$
\Gamma([g_1,g_2])=[\lambda_{g_1}g_1,\lambda_{g_2} g_2]=[g_1,g_2].
$$

\medskip

\begin{definition} \label{def:stand}
An automorphism of a Chevalley group $G(\Phi,R)$ or of its elementary subgroup $E(\Phi,R)$ over a local ring~$R$ is called \emph{standard}
if it is a composition of inner, diagonal, graph, ring and central automorphisms.

Similarly, an isomorphism between Chevalley groups $G(\Phi,R)$ and $G(\Phi,S)$ (or their elementary subgroups) over local rings $R$ and~$S$
is called \emph{standard} if it is a composition of inner, diagonal, graph, and central automorphisms of $G(\Phi,R)$ and a ring isomorphism between $G(\Phi,R)$ and $G(\Phi,S)$.
\end{definition}



Our main tool in the proof below is the following theorem.
It was obtained in a series of papers of the first author, see \cite{Bunina_main},~\cite{Bunina_recent}
and the references therein.

\begin{theorem}\label{standard_isom}
If $\Phi$ is an irreducible root system of rank $> 1$, $R$ and $S$ are  local rings $($with $1/2$ for $\Phi=\mathbf A_2, \mathbf B_\ell, \mathbf C_\ell, \mathbf F_4$ and with $1/3$ for $\mathbf G_2)$, then any isomorphism between $G(\Phi,R)$ and $G(\Phi,S)$  $($or between  $E(\Phi,R)$ and $E(\Phi,S))$ is standard.
\end{theorem}


\section{Proof of Theorem A} \label{sec:proof}

The general idea of the proof is to represent a given locally inner endomorphism in the
standard form, similarly to Theorem \ref{standard_isom}, prove that any ring homomorphism appearing in this
standard decomposition is an isomorphism, and then show that any diagonal automorphism is inner whereas
any graph, ring or central automorphism is identical. The proof is subdivided accordingly for the reader's
convenience. The subsections appearing below follow this subdivision.

Apart from Theorem  \ref{standard_isom}, our main tools are the Gauss and Bruhat decompositions. The first
is valid over the original ring $R$, and to apply the second, which holds for Chevalley groups over fields,
we go over to the group over the residue field and then use some lifting argument.

\subsection{Standard form of a locally inner endomorphism}\leavevmode
Throughout below our standing assumptions are that a
local ring $R$, a root system $\Phi$ and a Chevalley group $G(\Phi,R)$ 
satisfy all the conditions of Theorem~A.

Let $\varphi$ be a locally inner endomorphism of $G$ (or of its elementary subgroup $E$).

The following lemma can be viewed as an analogue of  Theorem~\ref{standard_isom}.

\begin{lemma} \label{lem:comp}
With the notation as above, $\varphi$ can be decomposed
into a composition of an inner, diagonal, graph and central automorphisms,
and a ring endomorphism of $G$ (or $E$):
\begin{equation} \label{eq:comp}
\varphi=\rho \circ \Lambda \circ \delta \circ d\circ i_g.
\end{equation}
\end{lemma}

\begin{proof}
Since $\varphi$ is injective, its image $G_0:=\varphi(G)$ is a subgroup of $G$ which is isomorphic to~$G$.
By a recent rigidity result of the first author \cite[Corollary~5]{Bu-bi}, any abstract group elementarily equivalent
to a Chevalley group $G(\Phi,R)$ is isomorphic to a Chevalley group $G(\Phi,S)$. {\it{A fortiori}}, this is true
for the group~$G_0$, which is isomorphic to~$G$, so that $G_0\cong G(\Phi,S)$ for some ring~$S$. We thus have
a decomposition of $\varphi$ into a composition of a monomorphism $\varphi_0$ and an isomorphism~$\varphi_1$:
$$
\varphi\colon G(\Phi,R) \overset{\varphi_1}{\longrightarrow} G(\Phi,S) \overset{\varphi_0}{\longrightarrow} G(\Phi,R).
$$
By Theorem~\ref{standard_isom} applied to $\varphi_0$ and~$\varphi_1$, we conclude that
$\varphi$ is a composition of inner, diagonal, graph and central automorphisms, and a ring endomorphism of $G(\Phi,R)$.

The same arguments apply to $E$ because in our set-up $E$ is a characteristic subgroup of $G$~\cite{Vas},
and the assertion of \cite[Corollary~5]{Bu-bi} is valid for~$E$.

The lemma is proven.  \end{proof}



\begin{remark} \label{rem:psi}
Note that since $\varphi$ is locally inner, so is $\psi:=\varphi \circ i_{g^{-1}}= \rho \circ \Lambda \circ \delta \circ d$.
\end{remark}

\begin{remark} \label{rem:coHopf}
The argument in the proof of Lemma \ref{lem:comp} is a sort of overkill.
It would be interesting to produce a more direct reasoning.

It however can be used
for describing co-Hopfian Chevalley groups over local rings. Recall that a group $G$
(or a ring $R$) is called \emph{co-Hopfian} if any its injective endomorphism is an isomorphism.
Any field $k$ which is not co-Hopfian (say, $k=\mathbb F_q(t)$ whose Frobenius endomorphism is
injective but not surjective) gives rise to a Chevalley
group $G(k)$ which is not co-Hopfian either: indeed, any injective non-surjective endomorphism of $k$
induces an injective non-surjective endomorphism of $G(k)$. The following statement can be proven in the
same way as Lemma \ref{lem:comp}.
\end{remark}

\begin{proposition} \label{prop:co-Hopf}
Let $R$ and $\Phi$ be as in Theorem A. Then $G(\Phi,R)$ is co-Hopfiam if and only if
so is $R$.
\end{proposition}

It is worth noting that the co-Hopfian property is often related to other rigidity properties, see, e.g.,
\cite{Be},~\cite{ALM}.

\subsection{Ring endomorphism is an automorphism}\leavevmode
We continue the proof of Theorem A.

\begin{lemma} \label{lem:ring}
Let $\varphi$ be a locally inner endomorphism as in Lemma \ref{lem:comp}.
Then in decomposition~\eqref{eq:comp} the ring endomorphism $\rho$ is an
isomorphism.
\end{lemma}

\begin{proof}
We use the decomposition of $\psi$ from Remark \ref{rem:psi} in
place of \eqref{eq:comp}.

Any $h_\alpha(t)$, $\alpha\in \Phi$, $t\in R^*$, is invariant under diagonal and central automorphisms, hence we have
\begin{equation} \label{eq:ringaut}
\psi(h_\alpha(t))=\rho \circ \Lambda \circ \delta \circ d (h_\alpha(t))=\rho \circ \Lambda \circ \delta (h_\alpha(t))=\rho\circ \Lambda (h_{\delta(\alpha)}(t))=h_{\delta(\alpha)}(\rho(t)).
\end{equation}
If $\alpha$ and $\beta$ are of the same length, $h_{\alpha}(s)$ and $h_\beta(s)$ are conjugate for any $s$, therefore
the right-hand side of \eqref{eq:ringaut} is conjugate to $h_{\alpha}(\rho(t))$, As $\psi$ is locally inner, the left-hand side of
\eqref{eq:ringaut} is conjugate to $h_{\alpha}(t)$, so that $h_\alpha(t)$ and $h_\alpha(\rho(t))$ are conjugate. In the matrix form
(with respect to any representation of $G$) these two elements are  diagonal matrices with entries taken from the sets
$$
S:=\{ 1, t, t^{-1}, t^2,t^{-2}\}\text{ and } \rho(S)=\{ 1,\rho(t), \rho(t)^{-1}, \rho(t)^2, \rho (t)^{-2}\};
$$
for the root system $\mathbf G_2$ and a short root, one has to adjoin $\{ t^3, t^{-3}\}$ and $\{ \rho(t)^3, \rho(t)^{-3}\}$ to these sets (see~\cite{Steinberg}).


Since these diagonal matrices are conjugate, they have the same entries up to reordering. Therefore,
the sets $S$ and $\rho (S)$ coincide. In other words, $\rho$ is a permutation of $S$, so that $t$ lies in the image of $\rho$.
Since $t$ is an arbitrary invertible element and $R$ is generated by its invertible elements, $\rho$ is a surjective automorphism of~$G$, which proves the lemma.  \end{proof}









\subsection{Graph automorphism is identical}\leavevmode

\begin{lemma} \label{lem:graph}
Let $\varphi$ be a locally inner endomorphism as in Lemma \ref{lem:comp}.
Then in decomposition~\eqref{eq:comp} the graph automorphism $\delta$ is
identical.
\end{lemma}

\begin{proof} As in the previous section, we replace decomposition \eqref{eq:comp} with the
decomposition of a locally inner automorphism $\psi$ of the Chevalley group $G(\Phi,R)$
into the composition $\rho\circ \Lambda \circ \delta \circ d$ of ring, central, graph and
diagonal automorphisms (recall that $\rho$ is an isomorphism by Lemma~\ref{lem:ring}).
Note that $\psi$ induces an automorphism $\overline \psi$ of the Chevalley group $G(\Phi,R/\Rad R)=G(\Phi,k)$
over the residue field $k$ of~$R$ since its components $\rho, \Lambda, \delta$ and~$d$ induce, respectively

(1) the ring (field) automorphism $\overline \rho$ of $k$ and $G(\Phi,k)$ since any automorphism of~$R$ maps the radical $\Rad R$ into itself;

(2) the central automorphism $\overline \Lambda$ of $G(\Phi,k)$;

(3) the graph automorphism $\overline \delta$ of $G(\Phi,k)$ induced by the same $\delta\in \Aut(\Phi)$;

(4) the diagonal automorphism $\overline d$ of $G(\Phi,R)$, which maps every $x_{\alpha_i}(1)$, $i=1,\dots, \ell$, to $x_{\alpha_i}(\overline d_i)$.

Since $\psi$ is locally inner, for any $g\in G(\Phi,R)$ we have $\psi(g)=xgx^{-1}$ for some $x\in G(\Phi,R)$.
Therefore for any $\overline g\in G(\Phi,k)$ we have
$$
\psi(\overline g)=\overline \psi(g)=\overline{xgx^{-1}}= \overline x \overline g \overline x^{-1},
$$
so that $\overline \psi$ is also a locally inner automorphism,

Consequently, if we prove that for any locally inner automorphism of a Chevalley group over an {\it arbitrary field} a graph automorphism
in the decomposition is identical, then we prove the same fact for a Chevalley groups over any local ring.
Thus in the sequel we assume that the base ring is a field. This allows us to use the Bruhat decomposition of Chevalley groups.

We consider several separate cases.

{\bf Case $\mathbf A_\ell$.}
The only non-trivial graph automorphism~$\delta$ of $\mathbf A_\ell$ maps each $x_{\alpha_i}(t)$, $1\leqslant i\leqslant \ell$, to $x_{\alpha_{\ell-i}}(t)$.
For
$$
g=x_{\alpha_1}(1) x_{\alpha_2}(1)\dots x_{\alpha_{\ell-1}}(1)\in G(\mathbf A_\ell, k),
$$
we have
$$
\psi(g)=x_{\alpha_\ell}(a_\ell)x_{\alpha_{\ell-1}}(a_{\ell-1})\dots x_{\alpha_2}(a_2)=x_{\alpha_2}(a_2)\dots x_{\alpha_{\ell-1}}(a_{\ell-1})x_{\alpha_\ell}(a_\ell) x_{\alpha_{\ell+1}}(a_{\ell+1})\dots x_{\alpha_m}(a_m),
$$
where $a_2,\dots, a_{\ell}\in k^*$, $a_{\ell+1},\dots,a_m\in k$.

Suppose that $g$ and $\psi(g)$ are conjugated by some $x\in G(\Phi,k)$ with the Bruhat decomposition
$$
x=t u_1 \mathbf w u_2,\quad t\in T(\mathbf A_\ell, k),\ u_1,u_2 \in U(G),\ \mathbf w\in W(G).
$$
This means that
$$
\mathbf w(x_{\alpha_1}(1) x_{\alpha_2}(1)\dots x_{\alpha_{\ell-1}}(1))^{u_2}\mathbf w^{-1}=(x_{\alpha_2}(a_2)\dots x_{\alpha_{\ell-1}}(a_{\ell-1})x_{\alpha_\ell}(a_\ell) x_{\alpha_{\ell+1}}(a_{\ell+1})\dots x_{\alpha_m}(a_m))^{u_1^{-1}t^{-1}},
$$
or, equivalently, 
$$
(x_{\alpha_1}(1)\dots x_{\alpha_{\ell-1}}(1)x_{\alpha_{\ell+1}}(b_{\ell+1})\dots x_{\alpha_m}(b_m))^{\mathbf w}=x_{\alpha_2}(c_2)\dots x_{\alpha_\ell}(c_\ell)x_{\alpha_{\ell+1}}(c_{\ell+1})\dots x_{\alpha_m}(c_m),
$$
where $c_2,\dots, c_\ell\in k^*$.

Since the right-hand side belongs to $U$, so does the left-hand side, therefore
$$
\mathbf w (\alpha_1),\dots , \mathbf w(\alpha_{\ell-1})\in \Phi^+.
$$
Since under the action of~$\mathbf w\in W$ only one simple root $\alpha_\ell$ can be mapped to a negative root,
either $\mathbf w=e$, or $\mathbf w=w_{\alpha_\ell}(1)$. It is clear that $\mathbf w=e$ is impossible since the left-hand and right-hand sides do not coincide.
If $\mathbf w=w_{\alpha_\ell}(1)$, then
$$
x_{\alpha_1}(1)\dots x_{\alpha_{\ell-1}}(1)x_{\alpha_{\ell+1}}(b_{\ell+1})\dots x_{\alpha_m}(b_m)=(x_{\alpha_2}(c_2)\dots x_{\alpha_\ell}(c_\ell)x_{\alpha_{\ell+1}}(c_{\ell+1})\dots x_{\alpha_m}(c_m))^{w_{\alpha_\ell}(1)},
$$
and the left-hand side of this equality belongs to $U$ while the right-hand side does not.

Therefore $g$ and $\psi(g)$ are not conjugate in $G(\mathbf A_\ell,k)$, contradiction.

\bigskip

{\bf Case  $\mathbf D_\ell$, $\ell\geqslant 5$.}
We apply the same idea as in the previous case. The only non-trivial graph automorphism $\delta$ swaps the last two simple roots,
so we have $\delta(x_{\alpha_{\ell-1}}(t))=x_{\alpha_\ell}(t)$, $\delta(x_{\alpha_\ell}(t))=x_{\alpha_{\ell-1}}(t)$, and for all other simple roots $\alpha_i\in \{ \alpha_1,\dots, \alpha_{\ell-2}\}$ we have $\delta(x_{\alpha_i}(t))=x_{\alpha_i}(t)$.

Take $g=x_{\alpha_1}(1)\dots x_{\alpha_{\ell-2}}(1)x_{\alpha_{\ell-1}}(1)$, then $\psi(g)=x_{\alpha_1}(a_1)\dots x_{\alpha_{\ell-2}}(a_{\ell-2})x_{\alpha_{\ell}}(a_\ell)$, where $a_1,\dots, a_\ell \in k^*$.
The same arguments as in the previous case show that $\delta (g)$ is not conjugate to $g$.

\bigskip

{\bf  Case  $\mathbf D_4$.} There are five different non-trivial graph automorphisms: three of them are similar to the previous case
(each of them swaps two leaves of the tree and preserves two other vertices of the Dynkin diagram),
and also there are two graph automorphisms of order 3 acting on the leaves of the tree). Since they are inverse to each other, we will consider only one of them:
$$
\delta_0(x_{\alpha_1}(t))=x_{\alpha_3}(t),\ \delta_0(x_{\alpha_2}(t))=x_{\alpha_2}(t),\ \delta_0(x_{\alpha_3}(t))=x_{\alpha_4}(t),\ \delta_0(x_{\alpha_4}(t))=x_{\alpha_1}(t).
$$
Taking $g=x_{\alpha_1}(1)x_{\alpha_2}(1)x_{\alpha_3}(1)$, we obtain
$$
\psi(g)=x_{\alpha_3}(1)x_{\alpha_2}(1)x_{\alpha_4}(1)=x_{\alpha_2}(a_2)x_{\alpha_3}(a_3)x_{\alpha_4}(a_4)x_{\alpha_5}(a_5)\dots x_{\alpha_m}(a_m),\quad a_2,a_3,a_4\in k^*,
$$
which is not conjugate to $g$, exactly as in the previous cases.


\bigskip

{\bf Case $\mathbf E_6$.} This case is treated similarly to $\mathbf A_\ell$ and $\mathbf D_\ell$. We have the simple roots $\alpha_1,\dots, \alpha_6$, where
$$
\delta(\alpha_1)=\alpha_6, \ \delta(\alpha_2)=\alpha_5,\ \delta(\alpha_3)=\alpha_3,\ \delta(\alpha_4)=\alpha_4,\ \delta(\alpha_5)=\alpha_2\text{ and }\delta (\alpha_6)=\alpha_1.
$$
So for $g=x_{\alpha_1}(1)\dots x_{\alpha_5}(1)$ we have
$$
\psi(g)=x_{\alpha_6}(1)\dots x_{\alpha_2}(1)=x_{\alpha_2}(a_2)\dots x_{\alpha_6}(a_6)x_{\alpha_7}(a_7)\dots x_{\alpha_m}(a_m),\quad a_2,\dots, a_6\in k^*,
$$
which is not conjugate to $g$ by the same arguments as in the case $\mathbf A_\ell$. \end{proof}




\subsection{Diagonal automorphism is inner}
By Lemmas \ref{lem:ring} and \ref{lem:graph}, we may and will assume that
our locally inner automorphism $\psi$ is the composition $\rho\circ \Lambda \circ d$ of ring, central and diagonal automorphisms.

\begin{lemma} \label{lem:diag}
The automorphism $d$ in the above decomposition is inner.
\end{lemma}

\begin{proof}
Take
$
g=x_{\alpha_1}(1)\dots x_{\alpha_\ell}(1)\in E(R),
$
then, since $\Lambda$ is identical on $E(R)$, we have
$$
\psi(g)=\rho (x_{\alpha_1}(d_1)\dots x_{\alpha_\ell}(d_\ell))=x_{\alpha_1}(\rho(d_1))\dots x_{\alpha_\ell}(\rho(d_\ell))=x_{\alpha_1}(d_1')\dots x_{\alpha_\ell}(d_\ell'),
$$
where $d_1',\dots, d_\ell'\in R^*$.

The elements $g$ and $\psi(g)$ are conjugated by some $x\in G(R)$.
Hence their images $\overline g$ and $\overline{\psi(g)}$ under factorization by $G(R,\Rad R)$ are conjugated in the quotient group $G(k)$, so that
$\overline g=x_{\alpha_1}(\overline 1)\dots x_{\alpha_\ell}(\overline 1)$ and $\overline{\psi(g)}=x_{\alpha_1}(\overline d_1')\dots x_{\alpha_\ell}(\overline d_l')$ are conjugated by
$$
\overline x=t u_1 \mathbf w u_2,\quad t\in T(k),\ u_1,u_2\in U(k),\ \mathbf w\in \mathbf W.
$$
We have
\begin{multline*}
t u_1 \mathbf w (x_{\alpha_1}(\overline 1)\dots x_{\alpha_\ell}(\overline 1))^{u_2} \mathbf w^{-1} u_1^{-1} t^{-1}=x_{\alpha_1}(\overline d_1')\dots x_{\alpha_\ell}(\overline d_l')\Longleftrightarrow \\
\Longleftrightarrow  (x_{\alpha_1}(\overline 1)\dots x_{\alpha_\ell}(\overline 1)\cdot x_{\alpha_{\ell+1}}(a_{\ell+1})\dots x_{\alpha_m}(a_m))^{\mathbf w}=(x_{\alpha_1}(\overline d_1')\dots x_{\alpha_\ell}(\overline d_\ell'))^{u_1^{-1}t^{-1}}\Longleftrightarrow\\
\Longleftrightarrow x_{\mathbf w(\alpha_1)}(\overline 1)\dots x_{\mathbf w(\alpha_l)}(\overline 1)\cdot x_{\mathbf w(\alpha_{l+1})}(a_{l+1})\dots x_{\mathbf w(\alpha_m)}(a_m)=\\
=x_{\alpha_1}(t_1\overline d_1')\dots x_{\alpha_\ell}(t_\ell\overline d_\ell')x_{\alpha_{\ell+1}}(b_{\ell+1})\dots x_{\alpha_m}(b_m),\text{ where }t_1,\dots, t_\ell\in k^*.
\end{multline*}
If $\mathbf w\ne e$, then there exist $j_1,\dots, j_k \in \{ 1,\dots, \ell\}$ such that $\mathbf w(\alpha_{j_1}),\dots, \mathbf w(\alpha_{j_k})\in \Phi^-$. Therefore the equality is impossible
since its right-hand side belongs to~$U$ and the left-hand side does not, so $\mathbf w=e$. Thus $\overline x=\overline t\overline u$, where $\overline t\in T(k)$ and $\overline u\in U(k)$.
This means that for the initial $x\in G(R)$, in the Gauss decomposition $x\in TUVU$ we have $x=tuv$, $t\in T(R)$, $u\in U(R)$ and $v\in V(R)\cap E(R, \Rad R)$.

Now in the initial group $G(R)$ we have
$$
( x_{\alpha_1}(1)\dots x_{\alpha_\ell}(1))^{(tuv)}= x_{\alpha_1}(d_1')\dots x_{\alpha_\ell}( d_\ell')
$$
or
$$
 x_{\alpha_1}(1)^v\dots x_{\alpha_\ell}(1)^v = x_{\alpha_1}(t_1\cdot d_1')\dots x_{\alpha_\ell)}(t_\ell\cdot  d_\ell')x_{\alpha_{\ell+1}}(b_{\ell+1})\dots x_{\alpha_m}(b_m) \in U(R),
$$
where $t_1,\dots, t_l\in R^*$ and $b_{\ell+1},\dots, b_m\in R$.

The left-hand side of this equality contains all $x_\alpha(1)$, $\alpha\in \Delta$, conjugated by
$$
v=x_{-\alpha_1}(c_1)\dots x_{-\alpha_m}(c_m),\quad c_1,\dots, c_m\in \Rad R.
$$
We know that for any $i\in \{ 1,\dots, \ell\}$ we have
$$
[x_{\alpha_i}(1),x_{-\alpha_1}(c_1)\dots x_{-\alpha_m}(c_m)]=\widetilde t x_{\alpha_i}(a)x_{-\alpha_i}(c_i')x_{-\alpha_{l+1}+\alpha_i}(c_i')\dots x_{-\alpha_m+\alpha_i}(c_m')\dots,
$$
where $-\alpha_k + \alpha_i\in \Phi^-$.

Therefore if $v\ne e$, then $x_{\alpha_1}(1)^v\dots x_{\alpha_\ell}(1)^v$ is necessarily of the form
$$
x_{\alpha_1}(a_1)\dots x_{\alpha_\ell}(a_\ell)x,\text{ where } x\in TV.
$$
Thus we have a contradiction and $v=e$. Consequently,
\begin{multline*}
 (x_{\alpha_1}(1)\dots x_{\alpha_\ell}(1))^{tu}= x_{\alpha_1}(d_1')\dots x_{\alpha_\ell}( d_\ell') \Longleftrightarrow\\
\Longleftrightarrow  (x_{\alpha_1}(1)\dots x_{\alpha_\ell}(1) x_{\alpha_{\ell+1}}(a_{\ell+1})\dots x_{\alpha_m}(a_m))^{t}=
x_{\alpha_1}(d_1')\dots x_{\alpha_\ell}( d_\ell) ,
\end{multline*}
where $a_{\ell+1},\dots, a_m\in R$.

This means that $[u,x_{\alpha_1}(1)\dots x_{\alpha_\ell}(1)]=e$ and
$$
x_{\alpha_1}(t_1)\dots x_{\alpha_\ell}(t_\ell) =
x_{\alpha_1}(d_1')\dots x_{\alpha_\ell}( d_\ell),\text{ so }t_1=d_1',\dots, t_\ell=d_\ell'.
$$
Therefore $t_1=\rho^{-1}(t)\in T(R)$ maps $x_{\alpha_1}(1),\dots, x_{\alpha_\ell}(1)$ to $x_{\alpha_1}(\rho^{-1}(d_1'))=x_{\alpha_1}(d_1),\dots, x_{\alpha_\ell}(d_\ell)$, respectively.
Thus $d=i_{t_1}$, so that $d$ is an inner automorphism.
\end{proof}





\subsection{Ring and central automorphisms are identical}
By Lemma \ref{lem:diag}, the automorphism $\rho \circ \Lambda=\psi \circ d^{-1}$ is locally inner.


\begin{lemma} \label{lem:ring-id}
The ring automorphism $\rho$ in the above decomposition is identical.
\end{lemma}

\begin{proof}
Assume the contrary. Then there exists $s\in R^*$ such that $\rho(s)=s_1\ne s$.

Take $g=h_{\alpha_1}(s)x_{\alpha_1}(1)\dots x_{\alpha_\ell}(1)$, then, taking into account that
the central automorphism acts trivially on $g$, we obtain
$$
g_1:=\rho \circ \Lambda (g)=h_{\alpha_1}(s_1)x_{\alpha_1}(1)\dots x_{\alpha_\ell}(1).
$$
If $g$ and $g_1$ are conjugated by an element $x=tu_1vu_2$, then in the same way as in the previous section we can prove that $v=u_2=e$ and $x=tu$. Then
\begin{multline*}
(h_{\alpha_1}(s)x_{\alpha_1}(1)\dots x_{\alpha_\ell}(1))^{tu}=\\
=(h_{\alpha_1}(s) h_{\alpha_1}(s)^{-1}uh_{\alpha_1}(s)u^{-1}x_{\alpha_1}(1)\dots x_{\alpha_\ell}(1)x_{\alpha_{\ell+1}}(a_{\ell+1})\dots x_{\alpha_m}(a_m))^t=\\
=h_{\alpha_1}(s) (u^{h_{\alpha_1}(s^{-1})}u^{-1}x_{\alpha_1}(1)\dots x_{\alpha_\ell}(1)x_{\alpha_{\ell+1}}(a_{\ell+1})\dots x_{\alpha_m}(a_m))^t=\\
=h_{\alpha_1}(s) u'u''x_{\alpha_1}(a_1)\dots x_{\alpha_\ell}(a_\ell)x_{\alpha_{\ell+1}}(a_{\ell+1}')\dots x_{\alpha_m}(a_m)=h_{\alpha_1}(s_1)x_{\alpha_1}(1)\dots x_{\alpha_\ell}(1).
\end{multline*}
Therefore $h_{\alpha_1}(s) u_1=h_{\alpha_1}(s_1) u_2$ and hence $s=s_1$.
\end{proof}



It remains to prove the following lemma.

\begin{lemma} \label{lem:cent}
Any locally inner central automorphism~$\Lambda$ is identical.
\end{lemma}

\begin{proof}
Recall that any central automorphism of $G$ is identical on the elementary subgroup $E$ since $E$ is the commutator subgroup of $G$.
By the way of contradiction, assume that $\Lambda$ is not identical on $G$.
Since for local rings $G=TE$, there exists $z\in T$ such that $\Lambda(z)=\lambda \cdot z=:z_1\ne z$.

Similarly to the previous case, take $g=z\cdot x_{\alpha_1}(1)\dots x_{\alpha_\ell}(1)$. Then
$\Lambda(g)=z_1\cdot x_{\alpha_1}(1)\dots x_{\alpha_\ell}(1)$. As $\Lambda$ is locally inner, $g$ and $\Lambda(g)$
are conjugated by $x=tu_1vu_2$. Then we get $v=u_2=e$, so that $z^{tu_1}= z\cdot x_{\alpha_1}(a_1)\dots x_{\alpha_m}(a_m)\ne z_1\cdot x_{\alpha_1}(1)\dots x_{\alpha_\ell}(1)$.
The obtained contradiction shows that $z=z_1$, hence $\Lambda$ is identical.
\end{proof}

\medskip


Theorem~A now follows from Lemmas \ref{lem:comp} and \ref{lem:ring}--\ref{lem:cent}. \qed

\section{Concluding remarks} \label{sec:final}

In this section we give a very brief overview of eventual generalizations of Theorem~A.

$\bullet$ The first natural move is to consider the items of Conjecture \ref{conj:main}
untouched by Theorem~A. We believe that the groups of rank $\geqslant 2$ mentioned there are \Sha-rigid. To
establish this property, in the relevant cases one could try using the descriptions of automorphisms
due to Borel and Tits~\cite{BT}, Weisfeiler \cite{We1}, \cite{We2}, Caprace~\cite{Ca}; a recent paper~\cite{BuGv}
will hopefully be useful as well. We are much less sure regarding groups of rank~$1$ where the presence of
non-standard automorphisms makes the problem far more tricky.

$\bullet$ Following the Gelfand principle (quoted from \cite{BR}), ``noncommutative algebra (properly understood)
is simpler than its commutative counterpart''. Thus the next step could be to consider linear groups of rank $> 1$ over
associative noncommutative rings. In such a set-up, one could try using the available description of automorphisms,
see \cite{GM}, \cite{Ze}, \cite{ABM}. Again, the case of groups of rank 1 requires much more efforts.
One should not be overoptimistic regarding \Sha-rigidity for the following reason. For any associative ring $R$
one can define locally inner endomorphisms and $\Sha (R)$ similarly to the group case and show that some rings
admit non-inner locally inner endomorphisms (actually, even automorphisms), see \cite{KO}. Any such endomorphism of~$R$
induces a non-inner locally inner endomorphism of $G(R)$, which leads to examples of $G$ with non-trivial $\Sha (G)$.
The most optimistic expectations can be formulated as follows: $G(R)$ is \Sha-rigid if and only if so is $R$.
(One can expect the same for the co-Hopfian property, in the spirit of Proposition \ref{prop:co-Hopf}.)
In such context, it would be interesting to understand the relationship between $\Sha (G(R))$ and $\Sha (R)$.
Ideally, one can hope to get a bijection between two sets, expressing the following reduction principle:
the group $G(R)$ is as $\Sha$-rigid as the ring~$R$. See \cite{SeTe}, \cite{Bu-bi} and~\cite{BuGv} for other instances of
a similar reduction principle.

$\bullet$ Refuting Conjecture \ref{conj:main} by exhibiting examples of Chevalley groups which are not \Sha-rigid
would perhaps be even more interesting. Some convincing instances in different contexts are listed below:

\begin{itemize}
\item[{$\bullet$}] Negative solution of Higman's isomorphism problem for integral group rings of finite groups, due to Hertweck \cite{He}, is based
on using groups $G$ with nontrivial $\Sha (G)$.
\item[{$\bullet$}] $\Sha (G)$ naturally appears
as a part of certain cohomology of Hopf algebras (called lazy cohomology,
see the paper of Guillot and Kassel \cite{GK}, or invariant
cohomology, see the paper of Etingof and Gelaki \cite{EtGe}).
\item[{$\bullet$}] For a finite group $G$, $\Sha (G)$ is related to another invariant with the origin
in mathematical physics (the so-called group of braided tensor autoequivalences of the Drinfeld centre of $G$).
This relationship was discovered by Alexei Davydov \cite{Da1}, \cite{Da2}, and each example of $G$ with
non-trivial $\Sha(G)$ may have conceptually important consequences.
\end{itemize}

\medskip

\noindent{\it {Acknowledgements.}} We thank Eugene Plotkin for his interest in this work and useful discussions.


\end{document}